\documentclass[12pt]{amsart}
\usepackage{amssymb, amsmath, amsfonts, microtype, cite, graphicx, tabu, yhmath}
\usepackage[colorlinks]{hyperref}
\usepackage[a4paper]{geometry}
\usepackage{xpatch}
\xpatchcmd{\itemize}{\def\makelabel}{\setlength{\itemsep}{5pt}\def\makelabel}{}{}
\xpatchcmd{\enumerate}{\def\makelabel}{\setlength{\itemsep}{5pt}\def\makelabel}{}{}
\usepackage{enumitem}
\setlist[enumerate,itemize]{topsep=4pt, parsep=3pt, itemsep=3pt, 
leftmargin=0.5cm, 
listparindent=20pt, labelsep=5pt, itemindent=0pt, partopsep=5pt}

\usepackage{mathtools}
\DeclarePairedDelimiter\abs{\lvert}{\rvert}
\def\bg#1{\bigl({#1}\bigr)}

\def\bgg#1{\biggl({#1}\biggr)}
\def\C{\mathbb{C}}

\def\R{\mathbb{R}}

\def\rmand{\quad\hbox{ and }\quad}

\numberwithin{equation}{section}
\numberwithin{figure}{section}

\DeclareMathOperator\sgn{\mathrm{sgn}}

\usepackage[capitalize]{cleveref}
\crefname{case}{Case}{Cases}
\crefname{case}{Case}{Cases}
\creflabelformat{case}{~\upshape#2#1#3}
\crefname{cond}{Condition}{Conditions}
\Crefname{cond}{Condition}{Conditions}
\creflabelformat{cond}{~\upshape(#2#1#3)}
\crefname{ineq}{Ineq.}{Ineqs.}
\Crefname{ineq}{Inequality}{Inequalities}
\creflabelformat{ineq}{~\upshape(#2#1#3)}
\crefname{rec}{Recurrence}{Recurrences}
\Crefname{rec}{Recurrence}{Recurrences}
\creflabelformat{rec}{~\upshape(#2#1#3)}
\crefname{icond}{Condition}{Conditions}
\Crefname{icond}{Condition}{Conditions}
\creflabelformat{icond}{~\upshape#2#1#3}
\crefname{thm}{Theorem}{Theorems}
\Crefname{thm}{Theorem}{Theorems}
\creflabelformat{thm}{~\upshape#2#1#3}

\crefrangeformat{equation}{Eqs.\!~(#3#1#4) ---~(#5#2#6)}
\crefrangeformat{ineq}{Ineqs.\!~(#3#1#4) ---~(#5#2#6)}

\newtheorem{thm}{Theorem}[section]
\newtheorem{cor}[thm]{Corollary}
\newtheorem{lem}[thm]{Lemma}

\theoremstyle{definition}
\newtheorem{case}{Case}[thm]

\makeatletter \@addtoreset{equation}{section} \makeatother

\allowbreak
\allowdisplaybreaks

\title[]{Geometry of limits of zeros \\[5pt] 
of polynomial sequences of type $(1,1)$}

\author[D.G.L. Wang]{David G.L. Wang$^\dag$$^\ddag$}
\address{
$^\dag$School of Mathematics and Statistics, Beijing Institute of Technology, 102488 Beijing, P. R. China\\
$^\ddag$Beijing Key Laboratory on MCAACI, Beijing Institute of Technology, 102488 Beijing, P. R. China}
\email{david.combin@gmail.com}

\author[J.J.R. Zhang]{Jerry J.R. Zhang$^\dag$}
\address{
$^\dag$School of Mathematics and Statistics, Beijing Institute of Technology, 102488 Beijing, P. R. China}
\email{jrzhang.combin@gamil.com}
\thanks{
D.G.L. Wang is supported by the General Program of National Natural Science Foundation of China (Grant No.\ 11671037).}

\begin{document}

\begin{abstract}   
In this paper, we study the root distribution of some univariate polynomials satisfying a recurrence of order two with linear polynomial coefficients. We show that the set of non-isolated limits of zeros of the polynomials is either an arc, or a circle, or a ``lollipop'', or an interval. As an application, we discover a sufficient and necessary condition for the universal real-rootedness of the polynomials, subject to certain sign condition on the coefficients of the recurrence. Moreover, we obtain the sharp bound for all the zeros when they are real.
\end{abstract}

\maketitle

\section{Introduction}

Root distribution of polynomials in a sequence 
discover intensive information about the interrelations 
of the polynomials in the sequence, 
especially when the sequence satisfies a recurrence.
Stanley \cite{StaW}
provides some figures for the root distribution of some polynomials in a sequence 
arising from combinatorics. 

In the study of the root distribution of sequential polynomials,
both the real-rootedness and the limiting distribution of zeros of the polynomials 
receive much attention. 
Some evidence for the significance of real-rootedness of polynomials
can be found in Stanley~\cite[\S 4]{Sta00}.
Bleher and Mallison~\cite{BM06} consider
the zeros of Taylor polynomials, 
and the asymptotics of the zeros for linear combinations of exponentials.
Some study on certain ``zero attractor'' of particular sequences of polynomials
can be found in~\cite{BG07,GHR09}. 
The exploration of zero attractors of Appell polynomials has been
regarded as ``gems in experimental mathematics'' in \cite{BG08}.
Limiting distribution of zeros has been used to study the four-color theorem
via the chromatic polynomials initiated by Birkhoff \cite{Bir12},
which amounts to the nonexistence of a chromatic polynomial 
with a zero at the point $4$.
Beraha and Kahane~\cite{BK79} examine the limits of zeros 
for the sequence of chromatic polynomials of 
a special family of $3$-regular graphs, 
described as to consist of an inner and outer square separated by $n$ $4$-rings.
It turns out that the number $4$ is a limit of zeros of polynomials in this family.

Motived by the LCGD conjecture from topological graph theory,
Gross, Mansour, Tucker and the first author~\cite{GMTW16-01,GMTW16-10}
study the root distribution of polynomials satisfying the recurrence
\begin{equation}\label[rec]{rec:AB}
W_n(z) = A(z)W_{n-1}(z)+B(z)W_{n-2}(z),
\end{equation}
where the functions $A(z)$ and $B(z)$ are polynomials such that
one of them is linear and that the other is constant. 
They established the real-rootedness subject to some sign conditions 
of the coefficients of $A(z)$ and $B(z)$. 
Since the real-rootedness implies the log-concavity,
they confirm the LCGD conjecture for many graph families
whose genus polynomials satisfy \cref{rec:AB} with the sign conditions.
Orthogonal polynomials and quasi-orthogonal polynomials have closed relations with
\cref{rec:AB};
see Andrews, Richard and Ranjan~\cite{ARR99B} and Brezinski, Driver and Redivo-Zaglia~\cite{BDR04}.
Jin and Wang~\cite{JW17X} characterized the common zeros of polynomials $W_n(z)$ 
for general $A(z)$ and $B(z)$.

Following Gross et al.~\cite{GMTW16-01}, 
a sequence $\{W_n(z)\}_n$ of polynomials satisfying \cref{rec:AB} is said to be of type $(\deg A(z),\,\deg B(z))$.
It is normalized if $W_0(z)=1$ and $W_1(z)=z$. 
When $A(z)=az+b$ and $B(z)=cz+d$ are linear, \cref{rec:AB} reduces to
\begin{equation}\label[rec]{rec2:linear}
W_n(z) = (az+b)W_{n-1}(z)+(cz+d)W_{n-2}(z).
\end{equation}
Concentrating on the root distribution, and considering the polynomials defined by $(-1)^nW_n(-z)$,
one may suppose without loss of generality that $c\ge 0$.
We use a quadruple $(\sgn(a),\sgn(b),\sgn(c),\sgn(d))$,
each coordinate of which is either $+$ or~$-$ or $0$, 
to denote the combination of signs of the numbers $a,b,c,d$.

Gross et al.~\cite{GMTW16-01,GMTW16-10},
establish the real-rootedness 
for Cases $(+,*,0,-)$, $(0,+,+,+)$ and $(0,+,+,-)$,
where the symbol~$*$ indicates that the number $b$ might be of any sign.
In Case $(-,-,+,-)$,
Wang and Zhang~\cite{WZ17X--+-} establish the real-rootedness of all polynomials $W_n(z)$
for when $\Delta_g>0$, where $\Delta_g=(b+c)^2+4d(1-a)$.
In Case $(+,+,+,+)$, they \cite{WZ17X++++} show 
that every polynomial $W_n(z)$ is real-rooted if and only if $ad\le bc$.

According to Beraha, Kahane, and Weiss' result~\cite{BKW75,BKW78} on limits of zeros 
of polynomials satisfying \cref{rec:AB}, polynomials satisfying \cref{rec2:linear} have 
at most two isolated limits of zeros.
In this paper, 
we show that the set of non-isolated limits of zeros of polynomials satisfying \cref{rec2:linear}
is either an arc, or a circle, or a 	``lollipop'', or an interval.
As an application, we can show that in Case $(+,-,+,-)$, every polynomial is real-rooted 
if and only if $ad\le bc$.
Moreover, 
when the isolated limits are real, the zeros approach to them in an oscillating manner
in Cases $(0,+,+,+)$ and $(+,+,+,+)$, that is,
from both the left and right sides of the isolated limits,
while the convergence way is from only one side in Case $(+,-,+,-)$; see \cref{thm:rr:<}.

We should mention that the generating function of the normalized polynomials satisfying \cref{rec:AB} is
\[
\sum_{n\ge0}W_n(z)t^n=\frac{1+(z-A(z))t}{1-A(z)t-B(z)t^2}.
\]
In comparison,
the root distribution of the polynomials generated by the function
\[
\sum_{n\ge0}W_n(z)t^n=\frac{1}{1-A(z)t-B(z)t^2}
\]
has been investigated in \cite{Tra14}, 
in which Tran found an algebraic curve containing the zeros 
of all polynomials $W_n(z)$ with large subscript $n$.

This paper is organised as follows. 
After reviewing necessary notion and and notation,
we interpret Beraha et al.'s characterization for polynomials satisfying \cref{rec2:linear} in \cref{thm:lz}.
In \S\ref{sec:rr}, we provide a sufficient and necessary condition 
of real-rootedness in Case $(+,-,+,-)$, 
and the root distribution when they are real-rooted as an application of \cref{thm:lz}.

\section{Geometry of the limits of zeros}

Throughout this paper, 
we let $a,b,c,d\in\R$, $ac\ne0$, and let $\{W_n(z)\}_{n\ge 0}$
be a sequence of polynomials satisfying \cref{rec2:linear}.
Then the polynomial $W_n(z)$ has leading term $a^{n-1}z^n$.
For any complex number $z=re^{i\theta}$ with $\theta\in(-\pi,\pi]$,
we use the square root notation $\sqrt{z}$ to denote 
the number $\sqrt{r}e^{i\theta/2}$, which lies in the right half-plane $\theta\in(-\pi/2,\,\pi/2]$.
The general formula in \cref{lem:00}
is the base of our study, which can be found in~\cite{GMTW16-01,GMTW16-10}. 

\begin{lem}\label{lem:00}
Let $A,B\in\C$.
Suppose that $W_0=1$ and $W_n=AW_{n-1}+BW_{n-2}$ for $n\ge 2$. Then 
\[
W_n=\begin{cases}
{\displaystyle {\alpha_+\lambda_+^n+\alpha_-\lambda_-^n}},&\textrm{ if $\Delta\neq0$},\\[5pt]
{\displaystyle {A+nh\over 2}\cdot\bgg{{A\over 2}}^{n-1}},&\textrm{ if $\Delta=0$},
\end{cases}
\]
for $n\ge 0$, 
where $h=2W_1-A$ and 
\[
\lambda_\pm=\frac{A\pm\sqrt{\Delta}}{2},\qquad
\alpha_\pm=\frac{\sqrt{\Delta}\pm h}{2\sqrt{\Delta}},
\qquad\text{with $\Delta=A^2+4B$}.
\] 
\end{lem}

Accordingly, we employ the notations
\begin{align*}
\Delta(z)
&=A(z)^2+4B(z)
=a^2z^2+(2ab+4c)z+(b^2+4d),\\[4pt]
h(z)&=2W_1(z)-A(z)=(2-a)z-b,\\[4pt]
\lambda_\pm(z)&=\frac{A(z)\pm\sqrt{\Delta(z)}}{2},\\
\alpha_\pm(z)&=\frac{\sqrt{\Delta(z)}\pm h(z)}{2\sqrt{\Delta(z)}},\\
%f(z)&=\abs{\lambda_+(z)}-\abs{\lambda_-(z)}
%=\frac{\abs{A(z)+\sqrt{\Delta(z)}}-\abs{A(z)-\sqrt{\Delta(z)}}}{2},\\
g(z)&=-\alpha_+(z)\alpha_-(z)\Delta(z)
%=W_1(z)^2-A(z)W_1(z)-B(z)
=\frac{h^2(z)-\Delta(z)}{4}
=(1-a)z^2-(b+c)z-d.
\end{align*}
Denote by $x_A=-b/a$ and $x_B=-d/c$ the zeros of $A(z)$ and~$B(z)$ respectively. 
The function $\Delta(z)$ has two zeros
\[
x_\Delta^{\pm}=x_A+\frac{-2c\pm2\sqrt{\Delta_\Delta}}{a^2},
\]
where $\Delta_\Delta=c^2-a^2B(x_A)$ is the discriminant of $\Delta(z)$.
A number $z^*\in\C$ is a {\em limit of zeros} of the sequence $\{W_n(z)\}_n$ of polynomials
if there is a zero $z_n$ of $W_n(z)$ for each $n$ such that $\lim_{n\to\infty}z_n=z^*$. 

\begin{lem}[Beraha et al.~\cite{BKW75}]\label{lem:BKW}
Under the non-degeneracy conditions 
\begin{enumerate}[label=\emph{(N-\roman*)}]
\item\label[icond]{cond:rec2}
the sequence $\{W_n(z)\}_n$ does not satisfy a recurrence of order less than two,
\item\label[icond]{cond:f<>0}
$\lambda_+(z)\ne\omega\lambda_-(z)$ for some $z\in\C$ and some constant $\omega$ such that $|\omega|=1$,
\end{enumerate}
a number $z$ is a limit of zeros if and only if it satisfies one of the following conditions:
\begin{enumerate}[label=\emph{(C-\roman*)}]
\item\label[icond]{cond:-}
$\alpha_-(z)=0$ and $\lambda_+(z)<\lambda_-(z)$;
\item\label[icond]{cond:+}
$\alpha_+(z)=0$ and $\lambda_+(z)>\lambda_-(z)$;
\item\label[icond]{cond:=}
$\lambda_+(z)=\lambda_-(z)$.
\end{enumerate}
\end{lem}
A limit $z$ of zeros is said to be {\em non-isolated} if it satisfies \cref{cond:=},
and to be {\em isolated} if it satisfies \cref{cond:-} or \cref{cond:+}.
We denote the set of non-isolated limits of zeros of the polynomials $W_n(z)$ by $\clubsuit$,
and denote the set of isolated limits of zeros by $\spadesuit$.
The clover symbol $\clubsuit$ is adopted
for the leaflets of a clover are not alone, while the spade symbol $\spadesuit$ 
appearing as a single leaflet represents isolation in comparison.

\begin{thm}\label{thm:lz}
Let $a,b,c,d\in\R$ and $ac\ne0$.
Let $\{W_n(z)\}_n$ be a sequence of polynomials satisfying \cref{rec2:linear}
with $W_0(z)=1$ and $W_1(z)=z$.
Then the sets of isolated and non-isolated limits of zeros of $\{W_n(z)\}_n$
are respectively
\begin{align*}
\spadesuit&=\{z\in\C\colon g(z)=0,\,\Re\bg{A(z)\overline{h(z)}}<0\}\rmand\\
\clubsuit&=\begin{cases}
\wideparen{x_\Delta^-x_Ax_\Delta^+},&\text{if $\Delta_\Delta<0$};\\[5pt]
C_0,&\text{if $\Delta_\Delta=0$};\\[5pt]
J_\Delta\cup C_0,&\text{if $\Delta_\Delta>0$ and $B(x_A)>0$};\\[5pt]
J_\Delta,&\text{if $\Delta_\Delta>0$ and $B(x_A)\le0$};
\end{cases}
\end{align*}
where $\overline{z}$ denotes the complex conjugate of $z$,
$\wideparen{x_\Delta^-x_Ax_\Delta^+}$ 
stands for the circular arc connecting the points $x_\Delta^-$ and $x_\Delta^+$, 
through the point $x_A$,
\[
C_0=\{z\in\C\colon\abs{z-x_B}=\abs{x_A-x_B}\}
\]
is the circle with center $x_B$ and radius $\abs{x_A-x_B}$, and
\[
J_\Delta=\{x\in\R\colon x_\Delta^-\le x\le x_\Delta^+\}
\]
is an interval.
\end{thm}

\begin{proof}
\Cref{cond:rec2} is satisfied since otherwise one would have 
$W_n(z)=z^n$ for each $n$,
contradicting the fact $W_2(z)=az^2+(b+c)z+d$.
\Cref{cond:f<>0} holds true since $|\lambda_-(x)|\ne|\lambda_+(x)|$ for sufficiently large real number $x$.

Suppose that $z\in\spadesuit$. From definition, we have 
$\alpha_-(z)\alpha_+(z)=0$, which implies 
\[
0=g(z)=\frac{h^2(z)-\Delta(z)}{4}.
\]
Thus $\sqrt{\Delta(z)}\in\{\pm h(z)\}$.
If $\sqrt{\Delta(z)}=h(z)$, then $\alpha_-(z)=0$ from definition.
By \cref{lem:BKW}, we have $\lambda_+(z)<\lambda_-(z)$, i.e., $\Re\bg{A(z)\overline{h(z)}}<0$.
Along the same line we can handle the other case $\sqrt{\Delta(z)}=-h(z)$.

It is clear that $\{x_A,\,x_\Delta^-,\,x_\Delta^+\}\subseteq\clubsuit$.
Let $z=x+yi\in\clubsuit$ such that $A(z)\Delta(z)\ne0$, where $x,y\in\R$.
If $y=0$, then $z,\,A(z),\,\Delta(z)\in\R$. In this case, we can infer that
\[
\lambda_-(z)=\lambda_+(z)
\quad\iff\quad
\Delta(z)<0
\quad\iff\quad
\Delta_\Delta>0 \;\text{and}\;  x\in (x_\Delta^-,\,x_\Delta^+).
\]
Otherwise $y\ne0$.  
We can infer that 
\begin{align*}
\lambda_-(z)=\lambda_+(z)
&\iff
\text{the vectors $A(z)$ and $\sqrt{\Delta(z)}$ are orthogonal}\\
&\iff
\text{the vectors $A^2(z)$ and $\Delta(z)$ have opposite directions}\\
&\iff
\text{$A^2(z)$ and $B(z)$ have opposite directions, $\abs{A^2(z)}<\abs{4B(z)}$}\\
&\iff
\begin{cases}
\Re A^2(z)\cdot \Im B(z)=\Re B(z)\cdot \Im A^2(z)\\[5pt]
\Im A^2(z)\cdot \Im B(z)<0\\[5pt]
\abs{\Im A^2(z)}<4\abs{\Im B(z)}
\end{cases}\\
&\iff
\begin{cases}
\,(x-x_B)^2+y^2=(x_A-x_B)^2\\
(x-x_A)(x-x_A+2c/a^2)<0
\end{cases}\\
&\iff
z\in C_0\cap S_0\setminus \{x_A,\,x_\Delta^-,\,x_\Delta^+\},
\end{align*}
where 
$S_0=\{z\in\C\colon \abs{\Re z-x_A}\le\abs{2c/a^2},\,c\cdotp(\Re z-x_A)\le 0\}$
is the vertical strip with boundaries $\Re z=x_A$ and $\Re z=x_A-2c/a^2$.
It is clear that the boundary $\Re z=x_A$ intersects the circle $C_0$ at the point $x_A$.
To figure out the intersection of the other boundary with $C_0$,
we proceed according to the sign of $\Delta_\Delta$.

Suppose that $\Delta_\Delta<0$. Then $J_\Delta=\emptyset$ from definition, and
\[
\Re \bg{x_\Delta^\pm}=x_A-\frac{2c}{a^2}
\rmand
\Im \bg{x_\Delta^\pm}=\pm\frac{2\sqrt{-\Delta_\Delta}}{a^2}.
\]
It follows that
\[
\bg{x_\Delta^\pm-x_B}^2
=\bgg{x_A-\frac{2c}{a^2}-x_B}^2+\biggl(\frac{2\sqrt{-\Delta_\Delta}}{a^2}\biggr)^2
=(x_A-x_B)^2.
\]
Thus the points $x_\Delta^\pm$ lie on the intersection of the boundary $\Re z=x_A-2c/a^2$ and 
the circle $C_0$. Since the intersection contains at most two points, 
the points $x_\Delta^\pm$ consitute the intersection.
Hence the set $\clubsuit=C_0\cap S_0$ is the circular arc $\wideparen{x_\Delta^-x_Ax_\Delta^+}$.

When $\Delta_\Delta=0$, the points $x_\Delta^\pm=x_A-2c/a^2$ coincide with each other.
As a consequence,  we have $C_0\cap S_0=C_0$ and $\clubsuit=J_\Delta\cup C_0=C_0$.

Below we can suppose that $\Delta_\Delta>0$.
Note that 
\begin{equation}\label{B:xA}
B(x_A)=c(x_A-x_B).
\end{equation}
When $B(x_A)\le 0$, we claim that $C_0\cap S_0=\{x_A\}$.
Let $z\in C_0\cap S_0$.
If $c>0$, then $x_A\le x_B$ by \cref{B:xA}.
Since~$z\in C_0$, we have $\Re z\ge x_A$.
Since $z\in S_0$, we have $c(\Re z-x_A)\le 0$. 
Therefore, we infer that $\Re z=x_A$, and $z=x_A$ consequently.
Otherwise $c<0$. Then $x_A\ge x_B$ by \cref{B:xA}.
In this case, $z\in C_0$ implies $\Re z\le x_A$,
and $z\in S_0$ implies $\Re z\ge x_A$. Hence $z=x_A$ for the same reason. 
This proves the claim.
Since $\Delta(x_A)=4B(x_A)\le0$, we have $x_A\in J_\Delta$. 
Hence $\clubsuit=J_\Delta$.

When $B(x_A)>0$, we claim that $C_0\subset S_0$.
Let $z\in C_0$. One may show $c(\Re z-x_A)\le 0$
in the same fashion as when $B(x_A)<0$.
By geometric interpretation and the condition $\Delta_\Delta>0$, we deduce that
\[
|\Re z-x_A|\le (\text{the diameter of $C_0$})
=2|x_A-x_B|<|2c/a^2|.
\]
This proves the claim and hence $\clubsuit=J_\Delta\cup C_0$.
\end{proof}

We remark that $z\in\spadesuit$ if and only if $\overline{z}\in\spadesuit$.
Since $\Delta_\Delta\le 0$ implies $B(x_A)>0$,
the case ``$\Delta>0$ and $B(x_A)\le0$'' in \cref{thm:lz} 
can be reduced to ``$B(x_A)\le0$''.

\begin{cor}\label{cor:rr}
Let $a,b,c,d\in\R$ and $ac\ne0$.
Let $\{W_n(z)\}_n$ be a sequence of polynomials satisfying \cref{rec2:linear}
with $W_0(z)=1$ and $W_1(z)=z$.
If every polynomial $W_n(z)$ for large $n$ is real-rooted,
then $B(x_A)\le0$, and $\Delta\ge0$ as a consequence.
\end{cor}

\begin{proof}
Since every polynomial $W_n(z)$ for large $n$ is real-rooted,
we have $\spadesuit\cup\clubsuit\subset\R$.
By \cref{thm:lz}, we find either $\clubsuit=J_\Delta$, 
or $\clubsuit=C_0$ and $C_0$ degenerates to a single point.
In the former case, we find $B(x_A)\le0$.
In the latter case, we have $\Delta_\Delta=0$ and $x_A=x_B$,
which is impossible since otherwise 
\[
0=\Delta_\Delta=c^2-a^2B(x_A)=c^2,
\]
a contradiction.  This completes the proof.
\end{proof}

When $\clubsuit=J_\Delta\cup C_0$,
it turns out that the set $\clubsuit$ looks like a lollipop;
see \cref{fig:lollipop}.
\begin{figure}[h]
\includegraphics[width=7cm]{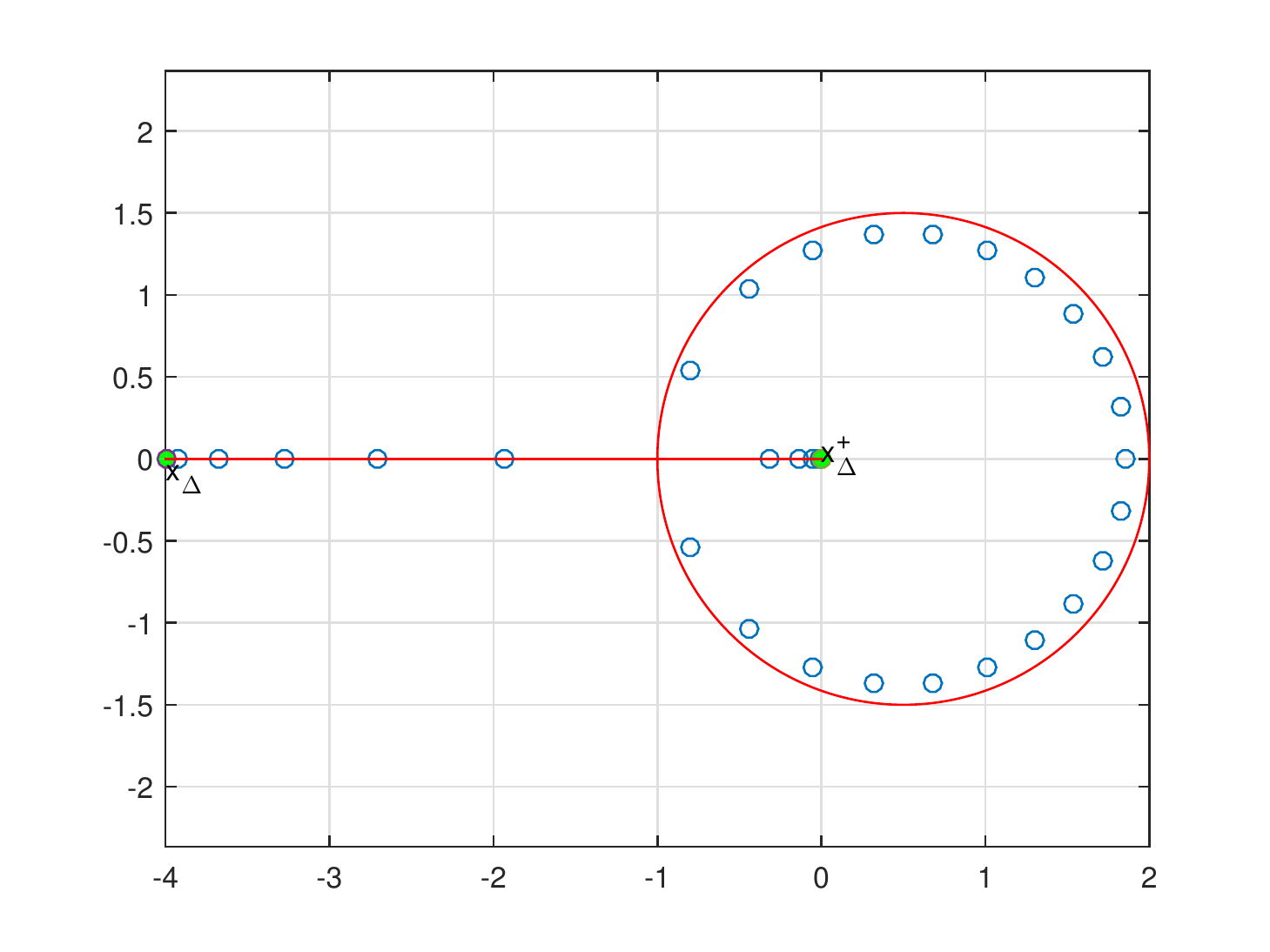}
\includegraphics[width=7cm]{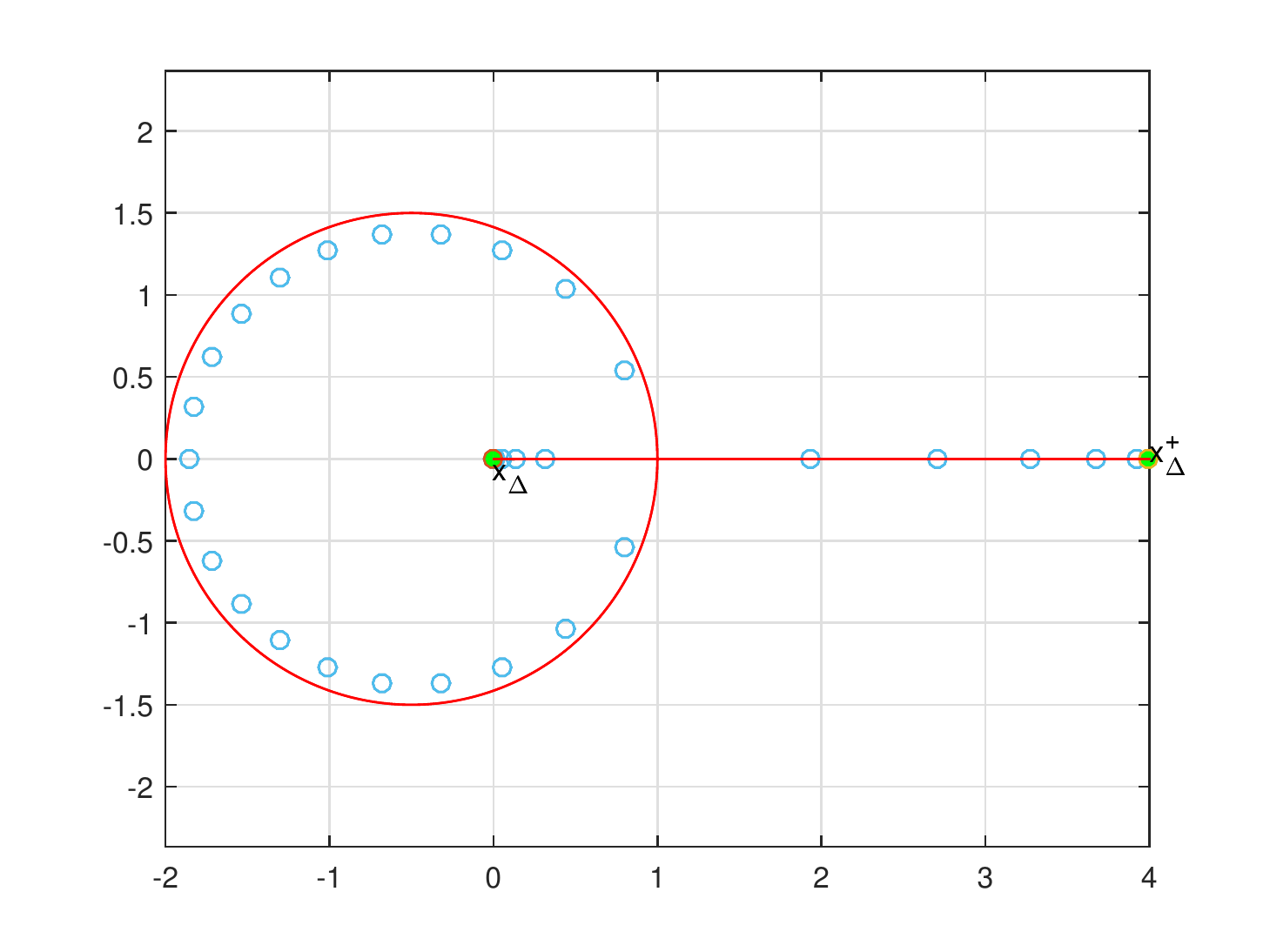}
\caption{The zero distribution of $W_{30}(z)$ for the parameters 
$(a,b,c,d)=(1,\,-2,\,2,\,-1)$ and $(a,b,c,d)=(1,\,2,\,-2,\,-1)$,
for each of which we have
$x_A=-2$, $x_B=-1/2$, and $B(x_A)=3$.}\label{fig:lollipop}
\end{figure} 

\begin{thm}\label{thm:dd>0:BxA>0}
Suppose $\Delta_\Delta>0$ and $B(x_A)>0$.†
Then $J_\Delta\cap C_0=\{2x_B-x_A\}$,
and the part of $J_\Delta$ outside the circle $C_0$ 
is longer than the part of $J_\Delta$ inside $C_0$.
\end{thm}
\begin{proof}
By \cref{thm:lz}, we have $\clubsuit=J_\Delta\cup C_0$.
First of all, denote $x_0=2x_B-x_A$ to be one of the two real points on $C_0$,
other than~$x_A$. Since 
\[
\Delta(x_0)=-\frac{4B(x_A)\Delta_\Delta}{c^2}<0,
\]
we have $x_0\in J_\Delta$.
Second, the centre of the circle $C_0$ is not on the interval $J_\Delta$ since 
$\Delta(x_B)=A^2(x_B)>0$.
It follows that $J_\Delta\cap C_0=\{x_0\}$.
Thirdly, note that
\begin{equation}\label{pf1}
x_0-\frac{x_\Delta^-+x_\Delta^+}{2}
=\frac{1}{c}\cdot \frac{2\Delta_\Delta}{a^2}.
\end{equation}
If $c>0$, then $x_B<x_A$ by \cref{B:xA}. 
It follows that $x_0<x_B$. Thus the interval $J_\Delta$ intersects
the circle $C_0$ from the left of $C_0$.
By \cref{pf1}, we have $x_0>(x_\Delta^-+x_\Delta^+)/2$.
Thus the part of $J_\Delta$ outside the circle $C_0$ 
is longer than the part of $J_\Delta$ inside.
The other case $c<0$ can be handled in the same way.
\end{proof}

\section{The interlacing zeros for Case $(+,-,+,-)$}\label{sec:rr}
Here is the main result of this section.

\begin{thm}\label{thm:rr}
Let $a,c>0$ and $b,d<0$.
Let $\{W_n(z)\}_n$ be a sequence of polynomials satisfying \cref{rec2:linear}
with $W_0(z)=1$ and $W_1(z)=z$.
Then $W_n(z)$ is real-rooted if and only if $x_A\le x_B$.
\end{thm}

The necessity part of \cref{thm:rr} can be seen directly from \cref{cor:rr}.
The sufficiency part will be handled for 
the case $x_A<x_B$ in \cref{thm:rr:<}, 
and for the case $x_A=x_B$ in \cref{thm:rr:=}. 
Throughout this section, we suppose that $x_A\le x_B$,
which implies that $\Delta_\Delta>0$ and $x_\Delta^\pm\in\R$. 
The zeros of the function $g(z)$ are
\[
x_g^\pm=\begin{cases}
\displaystyle
\frac{b+c}{2(1-a)}\pm\frac{\sqrt{\Delta_g}}{2\abs{1-a}},&\text{if $a\ne1$},\\[8pt]
\displaystyle -{d\over b+c},&\text{if $a=1$ and $b+c\ne0$},
\end{cases}
\]
where $\Delta_g=(b+c)^2+4d(1-a)$. 
We define two numbers $u$ and $v$ by
\begin{equation}\label{interval:uv}
(u,v)=\begin{cases}
(x_\Delta^-,\,x_\Delta^+),&\text{if $a<2$ and $F\le0$};\\[4pt]
(x_g^-,\,x_g^+),&\text{if $a>2$ and $F<0$};\\[4pt]
(x_g^+,\,x_\Delta^+),&\text{if $a<1$ and $F>0$};\\[4pt]
(x_g^-,\,x_\Delta^+),&\text{otherwise};
\end{cases}
\end{equation}
where $F=\Delta_g-\Delta_\Delta=d(a-2)^2+bc(2-a)+b^2$.
Note that $(u,v)=(x_\Delta^-,\,x_\Delta^+)$ if $a=1$ and $b+c=0$.
Furthermore, we have $u,v\in\R$ since $\Delta_g>\Delta_\Delta>0$ whenever $a\ge 2$ or $F>0$.
As will be seen in \cref{thm:rr:<,thm:rr:=}, we have $u<v$ and the interval $(u,v)$
is the best bound for the zeros of $W_n(z)$.

\subsection{Case $x_A<x_B$}
We determine the signs of $W_n(u)$ and $W_n(v)$ in \cref{lem:uv}.

\begin{lem}\label{lem:uv}
Let $a,c>0$ and $b,d<0$.
Let $\{W_n(z)\}_n$ be a sequence of polynomials satisfying \cref{rec2:linear}
with $W_0(z)=1$ and $W_1(z)=z$.
Suppose that $x_A< x_B$.
Then we have 
\begin{align}
&u\le x_\Delta^-<x_A<x_\Delta^+\le v<x_B,\label[ineq]{uv}\\[4pt]
&u<0<v,\label[ineq]{u0v}\\[4pt]
&W_n(u)(-1)^n>0,\label[ineq]{W:u}\\[4pt]
&W_n(v)>0,\label[ineq]{W:v}\rmand\\[4pt]
&\{u,v\}\subseteq\spadesuit\cup\clubsuit.\label{uv:club}
\end{align}
\end{lem}

\begin{proof}
The premise $x_A<x_B$ implies $\Delta(x_A)=4B(x_A)<0$.
It follows that
\begin{align}
&x_A\in (x_\Delta^-,\,x_\Delta^+),\qquad 
x_\Delta^+>0,\qquad
A(x_\Delta^+)>0>A(x_\Delta^-),\rmand\notag\\
&h(x_\Delta^+)=(2-a)x_\Delta^+-b\ge -b>0\qquad\text{if $a\le 2$}.\label[ineq]{h:xd2:+}
\end{align}
Since $\Delta(x_B)=A^2(x_B)>0$ and $x_\Delta^-<x_A<x_B$, we have $x_\Delta^+<x_B$.

To confirm Relation \eqref{uv:club}, by \cref{thm:lz},
it suffices to show that 
\begin{equation}\label[ineq]{dsr:uv:club}
A(x)h(x)<0,\qquad\text{for any $x\in\{u,v\}\backslash\{x_\Delta^-,\,x_\Delta^+\}$}.
\end{equation}
Let $x_h$ be the unique zero of the function $h(z)$ when $a\ne2$. Then $x_h=b/(2-a)$.
We proceed according to the definition of the numbers $u$ and $v$.

\begin{case}\label{case:a<2:F<=0}
$a<2$, $F\le0$ and $[u,v]=J_\Delta$. 
It is routine to compute that 
\begin{equation}\label{hh:xd}
h(x_\Delta^-)h(x_\Delta^+)=\frac{4F}{a^2}.
\end{equation}
Together with \cref{h:xd2:+}, we have $h(x_\Delta^-)\le 0$ and thus
\[
x_\Delta^-\le x_h=\frac{b}{2-a}<0,
\]
verifying \cref{u0v}. By \cref{lem:00}, we have
\begin{equation}\label{W:xd}
W_n(x_\Delta^\pm)
=\frac{A(x_\Delta^\pm)+nh(x_\Delta^\pm)}{2}\cdot\bgg{\frac{A(x_\Delta^\pm)}{2}}^{n-1},
\end{equation}
which implies \cref{W:u,W:v}.
\end{case}

\begin{case}\label{case:a>2:F<0}
$a>2$, $F<0$ and $[u,v]=[x_g^-,\,x_g^+]$.
Observe that 
\begin{equation}\label[ineq]{g:xd}
g(x_\Delta^\pm)=\frac{h^2(x_\Delta^\pm)}{4}\ge 0.
\end{equation}
Since the polynomial $g(z)$ is quadratic with leading coefficient negative,
we can derive all inequalities in \eqref{uv} except $v<x_B$.
Since $F<0$, we have $d(a-2)-bc<0$ and thus
\[
g(x_B)=\frac{-d}{c^2}\bg{(a-1)d-bc}<\frac{-d}{c^2}\bg{(a-2)d-bc}<0.
\]
Since $x_g^-<x_A<x_B$, we infer that $x_g^+<x_B$.

On the other hand,
by Vi\`eta's theorem,
we have 
\begin{equation}\label[ineq]{xg1xg2}
x_g^-x_g^+=\frac{d}{a-1},
\end{equation}
whose negativity verifies \cref{u0v}. 
By \cref{lem:00}, we have 
\begin{equation}\label{W:xg}
W_n(x_g^\pm)=(x_g^\pm)^n,
\end{equation}
which implies \cref{W:u,W:v}.
It is routine to compute that 
\begin{equation}\label{hh:xg}
h(x_g^-)h(x_g^+)=\frac{F}{a-1}\qquad\text{if $a\ne1$}.
\end{equation}
Thus $h(v)<0<h(u)$.
By \eqref{uv}, we have $A(u)<0<A(v)$.
This proves \cref{dsr:uv:club}.
\end{case}

\begin{case}\label{case:a<1:F>0}
$a<1$, $F>0$ and $[u,v]=[x_g^+,\,x_\Delta^+]$.
In view of \cref{W:xd,W:xg,h:xd2:+}, to confirm \cref{uv,u0v,W:u,W:v,dsr:uv:club},
we shall show that
\[
x_g^+\le x_\Delta^-,\qquad
x_g^+<0,\rmand
h(x_g^+)>0.
\]
In fact, we note that the polynomial $g(z)$ is quadratic with leading coefficient positive.
On the one hand, \cref{hh:xg} gives $x_h\in (x_g^-,\,x_g^+)$.
This confirms $h(x_g^+)>0$ immediately. 
By \cref{hh:xd}, we can deduce that $x_h<x_\Delta^-$, 
since otherwise one would have the absurd inequality
\[
0<x_\Delta^+<x_h=\frac{b}{2-a}<0.
\]
 Thus \cref{g:xd} implies $(x_g^-,\,x_g^+)\cap J_\Delta=\emptyset$.
Moreover,
the whole interval $(x_g^-,\,x_g^+)$ lies to the left of $J_\Delta$. 
This proves $x_g^+\le x_\Delta^-$.
On the other hand, by \cref{xg1xg2} we have $x_g^-x_g^+>0$.
Since $x_g^-<x_h<0$, 
we find $x_g^+<0$.
\end{case}

\begin{case}\label{case:remain}
For all remaining cases we have $[u,v]=[x_g^-,\,x_\Delta^+]$.
This time, to confirm \cref{uv,u0v,W:u,W:v,dsr:uv:club}, we shall show that
\[
x_g^-\le x_\Delta^-,\qquad
x_g^-<0,\qquad
h(x_\Delta^+)\ge0,\rmand
h(x_g^-)>0.
\]
In fact, when $a=1$,
in view of \cref{case:a<2:F<=0}, we now have $F>0$ and thus $b+c<0$. 
Note that $g(z)=-(b+c)z-d$.
It follows from \cref{g:xd} that $x_g^-\le x_\Delta^-$.
Since $g(0)=-d>0$, we obtain $x_g^-<0$.
By \cref{h:xd2:+}, we have $h(x_\Delta^+)\ge0$. It is routine to compute that
\[
h(x_g^-)=x_g^--b=-\frac{d}{b+c}-b=-\frac{F}{b+c}>0.
\]
Now, in view of \cref{case:a<2:F<=0,case:a<1:F>0}, we have $a>1$.
Consequently, one may derive $J_\Delta\subseteq[x_g^-,\,x_g^+]$ and $x_g^-<0$ 
as in \cref{case:a>2:F<0}. 
We shall handle the two inequalities involving $h$ according to the value range of $a$.
If $a=2$, then the function $h(z)=-b$ reduces to a positive constant and we are done.
Now we can suppose that $a\ne 2$. 
\begin{enumerate}[leftmargin=20pt]
\item
If $a>2$, then
\[
h(x_\Delta^-)+h(x_\Delta^+)
=\frac{4}{a^2}\bg{(a-2)c-ab}>0.
\]
In view of \cref{case:a>2:F<0}, we have $F\ge 0$.
By \cref{hh:xd}, we have $h(x_\Delta^-)h(x_\Delta^+)\ge 0$.
Therefore, we infer that $h(x_\Delta^+)\ge 0$. 
Since the polynomial $h(z)$ is strictly decreasing and $x_g^-<x_\Delta^+$, we have $h(x_g^-)>h(x_\Delta^+)>0$.
\item
If $1<a<2$, by \cref{h:xd2:+}, it suffices to show that $h(x_g^-)>0$.
In view of \cref{case:a<2:F<=0}, we have $F>0$.
By \cref{h:xd2:+,hh:xd}, we have $h(x_\Delta^-)>0$ and $x_h<x_\Delta^-$.
By \cref{hh:xg}, we have $h(x_g^-)h(x_g^+)>0$. 
Since $J_\Delta\subseteq[x_g^-,\,x_g^+]$, we deduce that $x_h<x_g^-$, i.e., $h(x_g^-)>0$.
\end{enumerate}
\end{case}
This completes the proof.
\end{proof}

Let $X,Y\subset\R$ such that $|X|-|Y|\in\{0,1\}$.
We say that {\em $X$ interlaces $Y$}, 
if the elements $x_i$ of $X$ and the elements $y_j$ of $Y$ can be arranged so that 
$x_1\le y_1\le x_2\le y_2\le\cdots$,
and that {\em $X$ strictly interlaces $Y$} 
if no equality holds in the ordering.
\Cref{lem:crt:itl} is Lemma 3.3 of \cite{GMTW16-10}, 
wherein used in a proof of the real-rootedness 
of polynomials $W_n(z)$ defined by \cref{rec2:linear}
with $a>0$, $b\in\R$, $c=0$ and $d<0$ by induction.

\begin{lem}[Gross et al.~\cite{GMTW16-10}]\label{lem:crt:itl}
Let $\{W_n(z)\}_n$ be a sequence of polynomials satisfying \cref{rec:AB}.
Let $m\ge0$ and $\alpha,\beta\in\mathbb{R}$.
Suppose that the polynomial $W_{m+2}(x)$ has degree $m+2$, and that
$B(x)<0$ for all $x\in R_{m+1}$,
\hbox{$W_{m} (\alpha)W_{m+2}(\alpha)>0$}, 
$W_{m} (\beta)W_{m+2}(\beta)>0$,
$|R_{m+1}|=m+1$,
$R_{m+1}\subset(\alpha,\beta)$, and
$R_{m+1}$ strictly interlaces~$R_{m}$.
Then we have 
$|R_{m+2}|=m+2$, $R_{m+2}\subset(\alpha,\beta)$, and
$R_{m+2}$ strictly interlaces $R_{m+1}$.
\end{lem}

Now we are in a position to show the real-rootedness with the interlacing property 
and the best bound of all zeros.

\begin{thm}\label{thm:rr:<}
Let $a,c>0$ and $b,d<0$ such that $x_A<x_B$.
Let $\{W_n(z)\}_n$ be a sequence of polynomials satisfying \cref{rec2:linear}
with $W_0(z)=1$ and $W_1(z)=z$.
Then every polynomial $W_n(z)$ is real-rooted. 
Denote by $R_n$ the zero set of~$W_n(z)$.
Then $R_n\subset (u,v)$, and the set $R_{n+1}$ strictly interlaces $R_n$.
Moreover, the bound $(u,v)$ is sharp, in the sense that both the numbers $u$ and $v$ are limits of zeros.
\end{thm}

\begin{proof}
We prove by induction with aid of \cref{lem:crt:itl} for $(\alpha,\beta)=(u,v)$.
Note that $R_1=\{0\}$. By \cref{lem:uv}, we have $u<0<v$. 
From definition, any singleton set strictly interlaces the empty set~$R_0$.
Now, we can suppose, for some $m\ge0$, 
that $|R_{m+1}|=m+1$, $R_{m+1}\subset (u,v)$, 
and $R_{m+1}$ strictly interlaces $R_m$.
Let $n\ge0$.
From \cref{rec2:linear}, 
every polynomial $W_n(z)$ is of degree $n$.
By \cref{lem:uv}, we have $B(x)<0$ for $x\in R_n$,
$W_n(u)W_{n+2}(u)>0$ and $W_{n}(v)W_{n+2}(v)>0$.
By \cref{lem:crt:itl}, we obtain the real-rootedness,
the bound $(u,v)$ and the strict interlacing property.
By \cref{thm:lz}, we have $\{x_\Delta^\pm\}\subseteq\clubsuit$.
By \cref{lem:uv}, we have $\{u,v\}\backslash\{x_\Delta^\pm\}\subseteq\spadesuit$. 
Hence both the numbers $u$ and $v$ are limits of zeros.
This completes the proof.
\end{proof}

We remark that the sharpness of the bound $(u,v)$ can be shown 
by using the totally different method demonstrated in the proof of Theorem 4.5 in \cite{GMTW16-10}.

\subsection{Case $x_A=x_B$}
Suppose that $x_A=x_B$.
%If $a\ge2$ in addition, then
%\[
%\Delta_g=(b+c)^2+4d(1-a)=c^2+b^2+(4-2a)d>c^2=\Delta_\Delta.
%\]
Then \cref{interval:uv} reduces to 
\[
u=\begin{cases}
x_\Delta^-,&\text{if $a<2$ and $F\le0$}\\[4pt]
x_g^+,&\text{if $a<1$ and $F>0$}\\[4pt]
x_g^-,&\text{otherwise}
\end{cases}
\qquad\rmand\qquad
v=x_\Delta^+=x_A=x_B.
\]
In an analogue with \cref{lem:uv}, we have \cref{lem:uv:=}.
\begin{lem}\label{lem:uv:=}
Let $a,c>0$ and $b,d<0$.
If $x_A=x_B$, then 
$u\le x_\Delta^-$,
$u<0$,
$W_n(u)(-1)^n>0$,
and $u\in\spadesuit$ as if $u\ne x_\Delta^-$.
\end{lem}
\begin{proof}
Same to the proof of \cref{lem:uv}.
\end{proof}
Now we can demonstrate the root distribution of the polynomials $\{W_n(z)\}$.
\begin{thm}\label{thm:rr:=}
Let $a,c>0$ and $b,d<0$ such that $x_A=x_B$.
Let $\{W_n(z)\}_n$ be a sequence of polynomials satisfying \cref{rec2:linear}
with $W_0(z)=1$ and $W_1(z)=z$.
Then the function $U_n(z)=W_n(z)/A^{\lfloor n/2\rfloor}(z)$ is a polynomial,
with all its zeros lying in the interval $(u,\,x_B)$. Moreover, the interval $(u,\,x_B)$ is sharp 
in the sense that both the numbers $u$ and $x_B$ are limits of zeros of the polynomials $U_n(z)$.
\end{thm}

\begin{proof}
By \cref{rec2:linear}, the functions $U_n(z)$ satisfy the recurrence
\begin{equation}\label[rec]{rec:U}
U_n(z)=\begin{cases}
\displaystyle \qquad U_{n-1}(x)+c'\cdot U_{n-2}(x),&\quad\text{if $n$ is even},\\[4pt]
\displaystyle A(x)U_{n-1}(x)+c'\cdot U_{n-2}(x),&\quad\text{if $n$ is odd},
\end{cases}
\end{equation}
where $c'=c/a$,
with $U_0(z)=1$ and $U_1(z)=z$. 
It follows immediately that the function $U_n(z)$ is a polynomial of degree $\lceil{n/2}\rceil$.
Let $R_n'$ be the zero set of $U_n(z)$.

We shall show by induction that 
the zeros $z_j$ of $U_n(z)$ strictly interlaces the zeros~$x_j$ of $U_{n-1}(z)$ from the left,
in the interval $(u,\,x_B)$, i.e.,
\begin{equation}\label[rl]{interlacing:U}
\begin{cases}
u<z_1<x_1<z_2<\cdots<z_{\lceil{\frac n2}\rceil}<x_{\lceil{\frac{n-1}{2}}\rceil}<x_B,
&\text{if $n$ is even};\\
u<z_1<x_1<z_2<\cdots<z_{\lceil{\frac{n-1}{2}}\rceil}<x_{\lceil{\frac{n-1}{2}}\rceil}
<z_{\lceil{\frac{n}{2}}\rceil}<x_B,
&\text{if $n$ is odd}.
\end{cases}
\end{equation}
We make some preparations. 
First, by \cref{rec:U}, it is direct to show by induction that $U_n(x_B)>0$.
Second,
by \cref{lem:uv:=}, we have $u\le x_\Delta^-<x_\Delta^+=x_A$ and $W_n(u)(-1)^n>0$.
Therefore, we have $A(u)<0$ and thus
\[
U_n(u)(-1)^{\lceil{n/2}\rceil}>0.
\]
In particular, we have $U_2(u)<0$. Since $U_2(u)=z+c'$, we have $u<-c'<0<x_B$.
This checks the truth for $n=2$.
Let $n\ge3$. By induction hypothesis, the set $R_{n-1}'$ strictly interlaces $R_{n-2}'$ from the left.
Therefore, we have
\[
U_{n-2}(x_j)(-1)^{\lceil{n/2}+j}\rceil>0\qquad \text{for $j\le \lceil{(n-1)/2}\rceil$.}
\]
By \cref{rec:U}, the number $U_n(x_j)$ has the same sign as the number $U_{n-2}(x_j)$, that is,
$U_{n}(x_j)(-1)^{\lceil{n/2}+j}\rceil>0$.
By using the intermediate value theorem, we derive the desired \eqref{interlacing:U}.

Same to the proof of \cref{thm:rr:<}, 
one may show the minimality of the interval $(u,x_B)$ as a bound of the zeros of
polynomials $W_n(z)$. Note that $x_\Delta^-\ne x_\Delta^+$.
By \cref{thm:lz}, each point in the interval $J_\Delta$ is a limit of zeros of the polynomials~$W_n(z)$.
Therefore, each point in $J_\Delta$ is a limit of zeros of the polynomials $U_n(z)$,
and the interval $(u,x_B)=(u,x_\Delta^+)$ becomes the best bound of the union of zeros of 
all polynomials $U_n(z)$. 
This completes the proof. 
\end{proof}

\end{document}